\documentclass[12pt]{amsart}
\usepackage{amsmath, amsfonts, amsthm,amssymb, color, hyperref, enumerate, extarrows, cite, graphicx}

\newtheorem{theorem}{Theorem}[section]

\newtheorem{definition}[theorem]{Definition}
\newtheorem{cor}[theorem]{Corollary}
\newtheorem{lem}[theorem]{Lemma}
\newtheorem*{notation*}{Notation}
{\theoremstyle{definition}

}
\newtheorem{pro}[theorem]{Proposition}

\numberwithin{equation}{section}
{\theoremstyle{definition}
\newtheorem{remark}[theorem]{Remark}
}
\usepackage[T1]{fontenc}

\usepackage[hyperpageref]{backref}
	
\hypersetup{
	colorlinks   = true,
	citecolor    = magenta}

\newcommand{\N}{\mathcal{N}}

\DeclareMathOperator{\Grass}{Gr}
\newcommand{\Gr}{\Grass}
\newcommand{\D}{\mathcal{D}}
\newcommand{\C}{\mathbb{C}}
\newcommand{\R}{\mathbb{R}}

\newcommand{\Coreof}{\mathfrak{C}}
\newcommand{\ssubset}{\subset\joinrel\subset}

\title{The Levi $q$-core and Property ($P_q$)}

\author{Gian Maria Dall'Ara}
\address{Istituto Nazionale di Alta Matematica ``Francesco Severi"\\ Research Unit Scuola Normale Superiore\\
	Piazza dei Cavalieri, 7 - 56126, Pisa (Italy)}
\email{dallara@altamatematica.it}

\author{Samuele Mongodi}
\address{Dipartimento di Matematica e Applicazioni\\
	Università degli Studi di Milano-Bicocca\\
	Via Roberto Cozzi, 55 - 20126 Milano (Italy)}
\email{samuele.mongodi@unimib.it}

\author{John N. Treuer}
\address{Department of Mathematics\\ 
	University of California \\
	San Diego, La Jolla, CA 92093, USA}
\email{jtreuer@ucsd.edu}

\begin{document}

\maketitle

\begin{abstract}
We introduce the Grassmannian $q$-core of a distribution of subspaces of the tangent bundle of a smooth manifold. This is a generalization of the concept of the core previously introduced by the first two authors. In the case where the distribution is the Levi null distribution of a smooth bounded pseudoconvex domain $\Omega\subseteq \mathbb{C}^n$, we prove that for $1 \leq q \leq n$, the support of the Grassmannian $q$-core satisfies Property $(P_q)$ if and only if the boundary of $\Omega$ satisfies Property $(P_q)$. This generalizes a previous result of the third author in the case $q=1$. The notion of the Grassmannian $q$-core offers a perspective on certain generalized stratifications appearing in a recent work of Zaitsev.
\end{abstract}

\renewcommand{\thefootnote}{\fnsymbol{footnote}}
\footnotetext{\hspace*{-7mm} 
\begin{tabular}{@{}r@{}p{16.5cm}@{}}
& Keywords:  Property ($P_q$), Levi core, $\overline{\partial}$-Neumann problem, compactness\\
& Mathematics Subject Classification: Primary 32W05  Secondary: 32T99 


\end{tabular}

\noindent\thanks{The third author is supported in part by the NSF grant DMS-2247175 Subaward M2401689. }

\noindent \date
}

\section{Introduction}
Let $\Omega$ be a bounded pseudoconvex domain with $C^{\infty}$-boundary $b\Omega$, and for $q \in \{1, \ldots, n\}$, let $L^2_{(0, q)}(\Omega)$ denote the square-integrable $(0, q)$-forms on $\Omega$.  One of the guiding questions for significant research in the $\overline{\partial}$-Neumann problem is:~When is the $\overline{\partial}$-Neumann operator on $(0, q)$-forms, $N_q:L^2_{(0, q)}(\Omega) \to L^2_{(0, q)}(\Omega)$, compact?  The $q = 1$ case was notably studied by Catlin \cite{Cat84} who developed a potential theoretic condition, called Property ($P$), which when satisfied by the compact set $b\Omega$ guarantees that $N = N_1$ is compact. Consequently, $N_q$ is compact for all other values of $q$ since for $q < n$, the operator $N_{q + 1}$ is compact whenever $N_q$ is compact.  Property ($P$) was studied in the context of Choquet theory by Sibony \cite{Si87}, under the name of $B$-regularity, and later was generalized to Property ($P_q)$ for $q \in \{1,\ldots, n\}$ (see \cite{FuSt01}).

\begin{definition}[Property ($P_q)$]
Given a compact set $X\subseteq \C^n$, we say that $X$ satisfies Property $(P_q)$ if for any given $M>0$, there exist a neighborhood $U$ of $X$ and a $C^2$ function $\phi:U\to[0,1]$ such that for any $z\in U$ the sum of the $q$ smallest eigenvalues of the Hermitian matrix 
$$L_\phi(z)=\left(\frac{\partial^2\phi(z)}{\partial z_j\partial\bar{z}_k}\right)_{j, k}$$ 
is at least $M$.
\end{definition}
\noindent{}The definition of Property ($P$) is the same as that of Property ($P_1)$. For $q > 1$, it remains true that if $b\Omega$ satisfies Property $(P_q)$, then the $\overline{\partial}$-Neumann operator $N_q$ is compact, \cite{FuSt01, St10}. 

Dall'Ara and Mongodi introduced in \cite{DaMo21} the Levi core of a smooth pseudoconvex domain for studying the $\overline{\partial}$-Neumann operator $N_1$.  This was used by Treuer \cite{Tr22} to give a sufficient condition for when Property ($P$) holds on $b\Omega$.

\begin{theorem}[\protect{\cite[Theorem 1.1]{Tr22}}]\label{Treuer theorem}
Let $\Omega$ be a bounded pseudoconvex domain with $C^{\infty}$ boundary.  The support of the Levi core satisfies Property ($P$) if and only if $b\Omega$ satisfies Property ($P$).
\end{theorem}

We refer to \cite{DaMo21} for the definition of the Levi core and of its support. As a corollary of Theorem \ref{Treuer theorem}, if Property ($P$) holds on the support of the Levi core, then the $\overline{\partial}$-Neumann operator $N_1$, and $N_q$ for all $q \in \{1,\ldots, n\}$, is compact. The purpose of this note is to define and study a generalization of the Levi core, the Levi $q$-core, which can be used to study the $\overline{\partial}$-Neumann operator on the $(0, q)$-forms, $N_q$.  Our main theorem is an analogue of Theorem \ref{Treuer theorem} for the Levi $q$-core.  

\begin{theorem}\label{main theorem}
Let $\Omega$ be a bounded pseudoconvex domain with $C^{\infty}$ boundary.  The support of the Levi $q$-core satisfies Property ($P_q$) if and only if $b\Omega$ satisfies Property ($P_q$).
\end{theorem}

See Section \ref{q-core section} for the definition of the Levi $q$-core and of its support. The proof of Theorem \ref{main theorem} is also in Section \ref{q-core section}. In Section \ref{Zaitsev section}, we discuss a connection of the Levi $q$-core with certain generalized stratifications used by Zaitsev in the recent work \cite{Za24}, containing a novel approach to compactness in the $\overline\partial$-Neumann problem. Finally, we highlight the fact that in the proof of Theorem \ref{main theorem} we use the following result.

\begin{theorem}[\protect{\cite[Proposition 1.9]{Si87}} for $q=1$, \protect{\cite[Corollary 4.14]{St10}} for any $q$]\label{countable union theorem}\label{thm:Countable union of Pq} Let $X$ be compact and suppose that $X = \cup_{k=1}^{\infty} X_k$ where each $X_k$ is compact and satisfies Property ($P_q$).  Then $X$ satisfies Property $(P_q)$.
\end{theorem}

The $q = 1$ case was proved by Sibony \cite{Si87}.  In \cite{FuSt01}, Fu and Straube observed that the $q \geq 1$ cases follow essentially verbatim.  
In the monograph, \cite[Corollary 4.14]{St10}, an alternative proof for the $q \geq 1$ cases is given, but an error is made with the set $A$ defined therein, \cite{St25}.  Recently, the $q \geq 1$ cases of Theorem \ref{thm:Countable union of Pq} was used by Zaitsev in \cite[Proposition 1.17]{Za24} (see also Section \ref{Zaitsev section} below).
Since Theorem \ref{thm:Countable union of Pq} is crucial to the main theorem of this paper and since the $q > 1$ cases are of independent interest, we take the opportunity to give an exposition of its proof in Appendix \ref{Appendix section}.

\section{Grassmannian $q$-core, Levi $q$-core, and the proof of the main theorem}\label{q-core section}

In this section, we present a generalization of the core construction described in \cite[Section 2]{DaMo21}. We formulate the definition and the general results for real tangent distributions; however, they continue to hold in the case of complex tangent distributions.

\medskip

Let $M$ be a smooth manifold and let $q\leq \dim M$ be an integer. Given $p\in M$, we denote by $\Gr_q(TM)_p$ the Grassmannian of $q$-planes in $T_pM$.  The collection of these $q$-Grassmannians together with the projection $\pi$ onto $M$, constitute the $q$-Grassmannian bundle $\pi:\Gr_q(TM)\to M$.

\medskip

Given a set $\mathcal{F}\subseteq\Gr_q(TM)$, we denote its \emph{fiber over $p\in M$} as $\mathcal{F}_p$; that is,
$$\mathcal{F}_p=\mathcal{F}\cap \Gr_q(TM)_p\;.$$
Its \emph{support} is
$$S_{\mathcal{F}}=\{p\in M:\ \mathcal{F}_p\neq\emptyset\}\;.$$

\begin{remark}\label{rmk:support_closed}
   Since Grassmannian bundles have compact fibers, if $\mathcal{F}$ is closed as a set in $\Gr_q(TM)$, then $S_{\mathcal{F}}$ is closed in $M$.
\end{remark}

We say that $\mathcal{D}\subseteq TM$ is a distribution of tangent subspaces if $\mathcal{D}_p=\mathcal{D}\cap T_pM$ is a vector subspace of $T_pM$ for all $p\in M$. The \emph{$q$-Grassmannian distribution induced by $\mathcal{D}$}, denoted $\Gr_q(\mathcal{D})\subseteq \Gr_q(TM)$, is
defined as follows: for $p\in M$,
$$\Gr_q(\mathcal{D})_p:=\Gr_q(\mathcal{D}_p)\;;$$
that is, $\Gr_q(\mathcal{D})_p$ is the set of $q$-planes of $\mathcal{D}_p$, $\Gr_q(\mathcal{D}_p)$.  Notice that the fiber $\Gr_q(\D)_p$ is empty if and only if $\dim \D_p<q$.

\begin{lem}
    If $\D$ is closed in $TM$, then $\Gr_q(\D)$ is closed in $\Gr_q(TM)$.
\end{lem}

\begin{proof}
    Let $\{v_n\}_{n\in\mathbb{N}}\subseteq\Gr_q(\D)$ be a sequence of points that converges to $v\in\Gr_q(TM)$. Then $v_n\in \Gr_q(\D_{p_n})$ where $p_n\in M$ converges to some $p\in M$. 
    Let $V_n \subseteq T_{p_n}M$ and $V \subseteq T_pM$ be subspaces corresponding to $v_n$ and $v$ respectively.  In particular, $V$ is the set of all $w \in T_pM$ such that $(p,w)$ is a limit in $TM$ of a sequence of the form $(p_n, w_n)$ with $w_n \in V_n$. Since $(p_n,w_n)\in \mathcal{D}$ and $\D$ is closed, $(p, w)\in \D$ for all $w\in V$, i.e.~$v\in\Gr_q(\D)$.
    
\end{proof}

\begin{remark}\label{rem_dimlessk} It follows by Remark \ref{rmk:support_closed} that for a closed distribution $\mathcal{D}$, we have $ p\not\in S_{\Gr_q(\D)}$ if and only if there exists a neighborhood $U$ of $p$ such that $\dim \D_{p'}<q$ for all $p'\in U$.
  Cf.~also \cite[Proposition 2.3]{DaMo21}. 
\end{remark}

Given a closed set $\mathcal{F}\subseteq\Gr_q(TM)$, its \emph{Grassmannian derived set} is
$$\mathcal{F}'=\mathcal{F}\cap \Gr_q(TS_{\mathcal{F}})\;,$$
where $TS$ is the (``smooth Zariski'') \emph{tangent distribution} to the subset $S\subseteq M$, as defined in \cite[Definition 2.5]{DaMo21}. Since $\mathcal{F}$ is closed, $S_{\mathcal{F}}$ is a closed set in $M$ (by Remark \ref{rmk:support_closed}) and $TS_{\mathcal{F}}$ is a closed distribution supported on $S_{\mathcal{F}}$ (see \cite[Proposition 2.6]{DaMo21}).  Consequently, $\Gr_q(TS_{\mathcal{F}})$ is closed and so is the Grassmannian derived set $\mathcal{F}'$. Given an ordinal $\alpha$, we define
$$\mathcal{F}^\alpha=\left\{\begin{array}{ccl}\mathcal{F}&\textrm{if}&\alpha=0\\(\mathcal{F}^{\beta})'&\textrm{if}&\alpha=\beta+1\\
\displaystyle\bigcap_{\beta<\alpha}\mathcal{F}^\beta&\textrm{if}&\alpha\textrm{ is a limit ordinal.}\end{array}\right.$$

\begin{lem}\label{lem_backtotheroots}
    If $\alpha=\beta+1$, then
    $$\mathcal{F}^\alpha=Gr_q(TS_{\mathcal{F}^\beta})\cap\mathcal{F}.$$
    If $\alpha$ is a limit ordinal, then
    $$\mathcal{F}^\alpha=\bigcap_{\beta<\alpha}Gr_q(TS_{\mathcal{F}^\beta})\cap\mathcal{F}.$$
\end{lem}

\begin{proof}The proof is essentially the same as \cite[Lemma 2.5]{Tr22}, which proved the analogous result for distributions of the tangent bundle.\end{proof}

We extend the definition of the core of a distribution $\mathcal{D}$ of the tangent bundle $TM$ in \cite[Definition 2.10]{DaMo21} to closed subsets $\mathcal{F}$ of the $q$-Grassmannian bundle $\Gr_q(TM)$. As $\{\mathcal{F}^{\alpha}\}$ is a decreasing sequence of closed sets, there is a countable ordinal $\gamma$ such that $\mathcal{F}^\gamma=\mathcal{F}^{\gamma'}$ for all $\gamma'\geq \gamma$ (by \cite[Theorem 6.9]{Ke95}) and we define the \emph{Grassmannian $q$-core} of $\mathcal{F}$ to be 
$$\Coreof(\mathcal{F}):=\mathcal{F}^\gamma.$$ 
The \emph{Grassmannian $q$-core of a closed distribution} $\mathcal{D}\subseteq TM$ is defined as $$\Coreof_q(\mathcal{D}):=\Coreof(\Gr_q(\mathcal{D})).$$ 

\begin{remark}
If $\mathcal{D}$ is a closed distribution, then $\Gr_1(\mathcal{D})'=\Gr_1(\mathcal{D}')$, where $\mathcal{D}'$ is the derived distribution (defined in \cite{DaMo21}). It follows that $\Coreof(\Gr_1(\mathcal{D}))=\Gr_1(\Coreof(\mathcal{D}))$, that is, the Grassmannian $1$-core of $\mathcal{D}$ coincides with the $1$-Grassmannian (that is, the projectivization)
of the core of \cite{DaMo21}. In particular, the support of the core of a distribution from \cite{DaMo21} is the same as the support of its Grassmannian $1$-core. 
    \end{remark}

Let $\D$ be a closed distribution in $TM$ and let $\mathcal{F}=\Gr_q(\D)$. Consider the sequence of Grassmannian derived sets $\{\mathcal{F}^\alpha\}$ and define $S_{-1}=M$, $S_\alpha=S_{\mathcal{F}^\alpha}$. The sets $S_\alpha\setminus S_{\alpha+1}$ are locally closed, and disjoint. Let $A$ be the set of ordinals $\alpha\geq 0$ such that $S_\alpha\neq S_{\alpha+1}$. Then, as in \cite[Lemma 2.8]{Tr22} or \cite[equation (3.1)]{Tr22}, we have the decomposition
\begin{equation}\label{eq_decomp}M=(S_{-1}\setminus S_0)\cup \bigcup_{\alpha\in A} (S_\alpha\setminus S_{\alpha+1})\cup S_{\Coreof_q(\mathcal{D})}\;.\end{equation}
We note that \cite[Lemma 2.8]{Tr22} was proved for the Levi null distribution, but its proof remains true for any closed distribution $\mathcal{D}$.
\begin{pro}\label{pro_holdimlessq}
For $\alpha\in A\cup\{-1\}$, if a point $x\in M$ lies in $S_{\alpha}\setminus S_{\alpha+1}$, then there exists a neighborhood $U_x$ of $x$ and a manifold $F=F_x\subseteq U_x$ such that  $S_\alpha\cap U_x\subseteq F$ and 
\begin{equation}\label{dimensional inequality}
\dim  \D_y \cap T_yF < q, \hbox{ for all } y\in U_x.
\end{equation}
  
\end{pro}
    
\begin{proof}
    The case $\alpha=-1$ follows from Remark \ref{rem_dimlessk}. We now assume that $\alpha \in A$. By Lemma \ref{lem_backtotheroots}, if $x\in S_\alpha\setminus S_{\alpha+1}$ then 
    \begin{equation}\label{equation 1 in prop 3.6}
    \mathcal{F}_x\cap \Gr_q(T_xS_{\alpha})=\emptyset.
    \end{equation} By \cite[Proposition 2.6(c)]{DaMo21}, there exists a neighborhood $U$ of $x$ and a manifold $F\subseteq U$ such that $S_\alpha\cap U\subseteq F$ and $T_xF=T_xS_{\alpha}$.  Recall that $\mathcal{F} = Gr_q(\mathcal{D})$.  Plugging this into \eqref{equation 1 in prop 3.6}, we get
    $$
    \Gr_q(\D_x)\cap \Gr_q(T_xF)=\emptyset.
    $$
    This holds if and only if $\dim \D_x\cap T_xF<q$. Since the dimension of the fibers  of a closed distribution is an  upper semicontinuous function \cite[Proposition 2.3]{DaMo21}, this inequality holds in a neighborhood $U_x\subseteq U$ of $x$.
\end{proof}

\begin{remark}
As stated before, everything can be extended to complex distributions on a real manifold $M$, that is, subsets $\mathcal{D}$ of the complexified tangent bundle $\C TM$ whose fibers are complex linear. In this case, one has to replace tangent distributions with their complexifications, Grassmannians of $q$-planes with Grassmannians of complex $q$-planes (inside the complexified tangent bundle), and dimensions over $\R$ with dimensions over $\C$. In particular, \eqref{dimensional inequality} in Proposition \ref{pro_holdimlessq} becomes $\dim_\C \C T_yF\cap \D_y<q$ for all $y\in U_x$.
\end{remark}

We now specialize the above to the Levi null distribution $\mathcal{N}\subseteq T^{1,0}M$ on the boundary $M$ of a smooth bounded pseudoconvex domain $\Omega$. 

\begin{definition}
The \emph{Levi $q$-core} of $M$ is defined as $\Coreof_q(\mathcal{N})$, the Grassmannian $q$-core of the Levi null distribution $\mathcal{N}$.
\end{definition} We are ready to give the proof of Theorem \ref{main theorem}. We will utilize the following fact. 

\begin{pro}[\protect{cf.~\cite[Proposition 1.12]{Si87}} for $q=1$, \protect{\cite[Proposition 4.15]{St10}} for any $q$]\label{pro:hol dim<q implies Pq}Let $1\leq q \leq n$ and suppose that $S$ is a smooth submanifold of $b\Omega$ such that $\dim_\C(\C T_pS\cap\N_p)<q$ for all $p\in S$. Then any compact subset of $S$ satisfies Property $(P_q)$.
\end{pro}

\begin{proof}[{\bf Proof of Theorem \ref{main theorem}}]
    With the previous notation, fix $\alpha\in A\cup\{-1\}$.
    For every $x \in S_{\alpha}\setminus S_{\alpha + 1}$, let $U_x$ and $F_x$ be the neighborhood and manifold given by Proposition \ref{pro_holdimlessq}. Notice that $U_x$ can be taken to be disjoint from $S_{\alpha+1}$. There exists a subset $\{x_j\}_{j \in J_{\alpha}} \subseteq S_{\alpha}\setminus S_{\alpha+1}$, where $J_{\alpha}$ is a countable index set, such that
    $$
    S_{\alpha}\setminus S_{\alpha+1}\subseteq \bigcup_{j\in\mathbb{N}} U_{x_j}\;.
    $$
    We can find countably many open sets $\{V_{j, k}\}_{k \in K_j}$ compactly contained in $U_{x_j}$, where $K_j$ is a countable index set, that cover $U_{x_j}$.  Therefore, $\{V_{j, k}\}_{k \in K_j,\, j \in J_{\alpha}}$ covers $S^\alpha\setminus S^{\alpha+1}$. 
    Consider the compact sets $$
    X_{j, k}^\alpha=(S_{\alpha}\setminus S_{\alpha+1})\cap \overline{V}_{j, k}=S_{\alpha}\cap \overline{V}_{j, k}.
    $$
    Each $X_{j, k}^\alpha$ is contained in the manifold $F_{x_j}$, which, by construction, is such that $$
    \dim \C T_yF_{x_j}\cap \N_y<q \hbox{ for all } y\in F_{x_j}.
    $$
    Hence, by Proposition \ref{pro:hol dim<q implies Pq}, each $X_{j, k}^{\alpha}$ satisfies Property ($P_q$). Therefore, \eqref{eq_decomp} allows us to write
    $$b\Omega=\left(\bigcup_{\alpha\in A\cup\{-1\}}\bigcup_{j \in J_{\alpha}}\bigcup_{k \in K_j} X^{\alpha}_{j, k} \right) \cup S_{\Coreof_q(\N)}\;.$$
    By Theorem \ref{thm:Countable union of Pq}, if $S_{\Coreof_q(\N)}$ satisfies Property ($P_q$), then $M$ does as well. The converse is trivially true, and the thesis follows.
\end{proof}

\section{Connection to Zaitsev's generalized stratifications}\label{Zaitsev section}

Abstracting the property of the decomposition \eqref{eq_decomp} asserted by Proposition \ref{pro_holdimlessq}, we obtain the following definition. Our choice of terminology comes from \cite[Definition 1.14]{Za24}. 

\begin{definition}\label{def:countable_regularity} Let $M$ be a smooth manifold and let $\mathcal{D}\subseteq \C TM$ be a closed complex distribution. We say that $M$ is \emph{countably $q$-regular with respect to $\mathcal{D}$} if it admits a partition 
\begin{equation}\label{the stratification}
M = \cup_{\alpha\in J} Z_\alpha
\end{equation} satisfying the following conditions: \begin{enumerate}
	\item $J$ is countable,	
	\item each $Z_\alpha$ is locally closed, 
	\item for each $\alpha\in J$ and each $p\in Z_\alpha$ there exists a submanifold $F\subseteq M$ containing an open neighborhood of $p$ in the topology of $Z_\alpha$ and such that \begin{equation}\label{eq:intersection<q}
		\dim_\C\left(\mathcal{D}_y\cap\C T_yF\right) <q\qquad \forall y\in F.
		\end{equation} 
\end{enumerate}
\end{definition}

If $\mathcal{D}$ is the Levi null distribution on a real hypersurface $M\subseteq \C^n$, then the definition above is essentially Definition 1.14 in Zaitsev's paper \cite{Za24}. The only difference is that there $F$ is required to be a CR submanifold, a property that is not needed for the result we are about to prove.

\begin{theorem}\label{thm:countably_regular}
Let $\mathcal{D}$ be as in Definition \ref{def:countable_regularity}.The manifold $M$ is countably $q$-regular with respect to $\mathcal{D}$ if and only if 
the Grassmannian $q$-core of $\mathcal{D}$ is empty. 
\end{theorem}

\begin{proof} The ``if'' part has already been proved, see \eqref{eq_decomp} and Proposition \ref{pro_holdimlessq}. 

Let's prove the converse implication. Let $S:=S_{\mathfrak{C}_q(\mathcal{D})}$, which is a closed subset of $M$ with the property that for every $p\in S$ we have $\dim_\C\mathcal{D}_p\cap \C T_pS \geq q$. We argue by contradiction, assuming $M$ is countably $q$-regular with respect to $\mathcal{D}$ and $S$ is not empty. Let $W_\alpha=Z_\alpha\cap S$, where $Z_\alpha$ ($\alpha\in J$) is as in Definition \ref{def:countable_regularity}, and notice that $(\overline{W_\alpha})_{\alpha\in J}$ is a countable covering by closed sets of $S$. Since $S$ is locally compact and Hausdorff, it is a Baire space and therefore for at least one $\alpha$ the interior $U$ of $\overline{W_\alpha}$ in $S$ is not empty. Then $W_\alpha\cap U$ is dense in $U$. Since $Z_\alpha$ is locally closed in $M$, $W_\alpha$ is locally closed in $S$. Shrinking $U$, one may ensure that $W_\alpha\cap U$ is closed in $U$, and hence equal to $U$. Let $p\in U$. We have \begin{eqnarray*}\mathcal{D}_p\cap \C T_p S
&\subseteq&\mathcal{D}_p\cap \C T_p W_\alpha\\
&\subseteq&\mathcal{D}_p\cap \C T_p Z_\alpha\\
&\subseteq&\mathcal{D}_p\cap \C T_p F, 
\end{eqnarray*}
where $F$ is as in part 3 of the definition of countable $q$-regularity. This contradicts the fact that $\mathcal{D}_p\cap \C T_pS$ has dimension at least $q$. 
\end{proof}

The main thrust of Theorem \ref{thm:countably_regular} is that the existence of a ``stratification'' as in \eqref{the stratification} of Definition \ref{def:countable_regularity} is equivalent to the existence of a \emph{canonical} stratification, namely the one obtained iterating the Grassmannian derived set construction. Combining Theorem \ref{thm:countably_regular} and Theorem \ref{main theorem} and noting that the notion of countable $q$-regularity in \cite{Za24} implies countable $q$-regularity in the sense of Definition \ref{def:countable_regularity}, we recover Proposition 1.17 of \cite{Za24}.  This accomplishes proving the implication ``generalized stratifications $\Rightarrow$ Property ($P_q$)'' in that paper (cf. the diagram on page 5 of \cite{Za24}).

\begin{cor} (cf. \cite[Proposition 1.17]{Za24})
Let $\Omega\subseteq \C^n$ be a smooth bounded pseudoconvex domain with boundary $M$. Assume that $M$ is countably $q$-regular, in the sense of Definition \ref{def:countable_regularity} with respect to the Levi null distribution. Then $M$ satisfies Property ($P_q$) and hence the $\overline\partial$-Neumann operator $N_q$ is compact. 
\end{cor}

\appendix

\section{Proof of Theorem \ref{thm:Countable union of Pq}}\label{Appendix section}

 We begin by defining the continuous $q$-subharmonic functions.

\begin{definition}
For $q \in \{1,\ldots, n\}$, the continuous $q$-subharmonic functions on an open set $U$, $P_q(U)$, is the set of continuous functions on $U$ that are subharmonic on every $q$-dimensional affine subspace $\Omega \subseteq U$.  For compact $X \subseteq \mathbb{C}^n$, $P_q(X)$ denotes the closure in the uniform topology on $X$ of the functions $f$ such that there is a neighborhood $U_f$ of $X$ where $f \in P_q(U_f)$.
\end{definition}

We omit the proofs of the following basic facts.
\begin{lem}\label{closed under maximums}
    If $U$ is an open set and $\phi_1, \phi_2 \in P_q(U)$, then $\max\{\phi_1, \phi_2\} \in P_q(U)$. 
\end{lem}
\begin{lem}\label{closed under maximums for X compact}
Let $X$ be compact and $\phi_1, \phi_2 \in P_q(X)$.  Then $\max\{\phi_1, \phi_2\} \in P_q(X)$.
\end{lem}

\begin{lem}\label{Pq(X) is a closed space}
Suppose $X$ is compact and $\phi_j \in P_q(X)$ such that $\phi_j$ converges uniformly on $X$ to a function $f$.  Then $f \in P_q(X)$.
\end{lem}

\begin{lem}[Gluing Lemma]\label{gluing lemma}
Let $U$ be an open set and let $\omega$ be a non-empty proper open subset of $U$.  If $u \in P_q(U)$, $v \in P_q(\omega)$, and $max\{u, v\} = u$ in a neighborhood in $U$ of each $y \in \partial \omega \cap U$, then the formula
$$
w = \begin{cases}
    \max\{u, v\} & \omega
    \\
    u & U\setminus\omega
\end{cases}
$$
defines a function in $P_q(U)$.
 \end{lem}

As with plurisubharmonic and subharmonic functions, continuous $q$-subharmonic functions can be approximated by smooth ones.

 \begin{lem}\label{smooth approximation}
     Let $X$ be compact and $f \in P_q(X)$.  Then there exists a sequence of neighborhoods $U_j$ of $X$ and functions $f_j \in C^{\infty}(U_j)\cap P_q(U_j)$ which approach $f$ uniformly on $X$.
 \end{lem}
 \begin{proof}
     By definition, there exists a sequence of neighborhoods $V_j$ of $X$ and functions $f_j \in P_q(V_j)$ such that
     $
     \|f_j - f\|_{L^{\infty}(X)} \leq 2^{-j}.
     $
     By shrinking each $V_j$, we may suppose that $f_j \in L^{\infty}(V_j)$.  Let $\{\chi_{\epsilon}\}$
 be an approximation to the identity where each $\chi_{\epsilon}$ is radial.  For each $j$, select $\epsilon_j$ such that 
 $
U_j = V_j \cap \{z: \hbox{dist}(z, \partial V_j) > \epsilon_j\}
 $
 still contains $X$ and
 $
 \|\chi_{\epsilon_j} * f_j - f_j\|_{L^{\infty}(X)} \leq 2^{-j},
 $
 \cite[8.14 Theorem]{Fo99}.  Then, $ \chi_{\epsilon_j} *f_j \in P_q(U_j) \cap C^{\infty}(U_j)$ \cite[Proposition 1.2]{AhDi09} and approaches $f$ uniformly on $X$.
 \end{proof}

Given a compact set $X$ and $f\in C(X)$, below we define
$$\tilde{f}(z)=\sup\{\phi(z):\ \phi\in P_q(X)\textrm{ such that }\phi\leq f\textrm{ on }X\}.$$
A probability measure $\mu$ is called a $q$-Jensen measure on $X$ centered at $z \in X$ if it is supported on $X$ and $f(z) \leq \int_X f\, d\mu$ for all $f \in P_q(X)$.   Let $J_{q,z}(X)$ be the set of all such measures. By Edwards' Theorem \cite[(4.40), p.89]{St10} \cite[1.2 Theorem]{Ga78},
\begin{equation}\label{Edwards Theorem}
\tilde{f}(z)=\inf\left\{\int_X f\,d\mu:\ \mu\in J_{q,z}(X)\right\}.
\end{equation}

The $q$-Jensen boundary of $J_q(X)$ is the set of points where $J_{q, z} = \{\delta_z\}$, the set containing only the Dirac delta measure at $z$.  $X$ satisfies Property ($P_q$) if and only if $X = J_q(X)$, \cite[Section 4.5]{St10}.

\begin{lem} \label{tilde in Pq}
Let $f \in C(X)$.   If
$\tilde{f}(z)$
is continuous on $X$, then $\tilde{f} \in P_q(X)$.
 \end{lem}
\begin{proof}
    It suffices to show that there is a sequence of neighborhoods of $U_j$ of $X$ and functions $f_j$ defined on $U_j$ such that $f_j \in P_q(U_j)$ and $f_j$ converges uniformly to $\tilde{f}$ on $X$.  Consider the nonempty set
    $
    S = \{\lambda: \lambda \in P_q(X), \lambda \leq f\}.
    $
    For each $p \in X$, let $\lambda_p$ be a function in $S$ such that
    $
    0 \leq \tilde{f}(p) - \lambda_p(p) \leq {\epsilon \over 3}.
    $
    Let $\delta_p$ and $\tau_p$ be such that for all $q \in X$ with $|q - p| < \delta_p$ and $|q - p| < \tau_p$, we have
    $
    |\lambda_p(p) - \lambda_p(q)| < {\epsilon \over 3} \hbox{ and } 
    |\tilde{f}(p) - \tilde{f}(q)| < {\epsilon \over 3}, \hbox{ respectively.}
    $
    By compactness, there is a cover $\{S_k\} = \{X \cap \mathbb{B}(p_k, \min(\tau_{p_k}, \delta_{p_k})\}_{k =1}^N$ of $X$ such that for any $q \in S_k$,
    $
    0 \leq \tilde{f}(q) - \lambda_{p_k}(q) \leq \epsilon.
    $ Let $f_{\epsilon} = \max(\lambda_{p_1},\ldots,\lambda_{p_N})$. By Lemma \ref{closed under maximums for X compact}, $f_{\epsilon} \in P_q(X)$.  An elementary argument shows that $\tilde{f}$ is the uniform limit of functions $f_{\epsilon} \in P_q(X)$. By Lemma \ref{Pq(X) is a closed space}, $\tilde{f} \in P_q(X)$.
\end{proof}

Below $\mathbb{B}(p, r)$ will denote a ball centered at $z$ of radius $r$.
\begin{lem}\label{A Pq function which is close on the Jensen boundary}
    Let $X$ be compact and
$
X(p, r) = \overline{\mathbb{B}(p, r)} \cap X \subset \hbox{int }J_q(X).
$
There exists neighborhoods $U_m$ of $X$ and functions $\phi_m \in P_q(U_m) \cap C^{\infty}(U_m)$ such that
$$
|\phi_m(z) + C|z|^2| < {1 \over m}, \quad z \in \mathbb{B}(p, r) \cap U_m.
$$
\end{lem}
\begin{proof}
Let $f(z) = -C|z|^2$ and $\epsilon < {1 \over 4m}$. Since $X(p, r) \subseteq J_q(X)$, by Edwards' Theorem, \eqref{Edwards Theorem}, 
$
\tilde{f}(z) = -C|z|^2$ for all $z \in X(p, r)$.
Consider the nonempty set
$
S = \{\lambda: \lambda \in P_q(X), \lambda \leq -C|z|^2\}.
$
For each $z^{\prime} \in X(p, r)$, let $\lambda_{z^{\prime}}$ be a function in $S$ such that
$
0 \leq \tilde{f}(z^{\prime}) - \lambda_{z^{\prime}}(z^{\prime}) \leq \epsilon.
$
 Let $\delta_{z^{\prime}}$ and $\tau_{z^{\prime}}$ be such that for all $z^{\prime\prime} \in X(p, r)$ with $|z^{\prime\prime} - z^{\prime}| < \delta_{z^{\prime}}$ and $|z^{\prime\prime} - z^{\prime}| < \tau_{z^{\prime}}$, we have 
    $
    |\lambda_{z^{\prime}}(z^{\prime}) - \lambda_{z^{\prime}}(z^{\prime\prime})| \leq \epsilon \hbox{ and }
    |\tilde{f}(z^{\prime}) - \tilde{f}(z^{\prime\prime})| \leq \epsilon, \hbox{ respectively.}
    $
    Let $\{S_k\} = \{X(p, r) \cap \mathbb{B}((z^{\prime})_k, \min(\tau_{(z^{\prime})_k}, \delta_{(z^{\prime})_k})\}_{k =1}^N$ be a finite cover of $X(p, r)$.  For $z^{\prime\prime} \in S_k$,
    $
    0 \leq \tilde{f}(z^{\prime\prime}) - \lambda_{(z^{\prime})_k}(z^{\prime\prime})  \leq 3\epsilon.
    $
    Let $f_{m} = \max(\lambda_{(z^{\prime})_1},\ldots,\lambda_{(z^{\prime})_N})$. By Lemma \ref{closed under maximums for X compact}, $f_{m} \in P_q(X)$.
    Let $z \in X(p, r)$.  There is a $K \in \{1, \ldots, N\}$ such that $z \in S_K$.  Then
    $$
    0 \leq \tilde{f}(z) - f_{m}(z) = \tilde{f}(z) - \max(\lambda_{(z^{\prime})_1},\ldots,\lambda_{(z^{\prime})_N})(z)  \leq \tilde{f}(z) - \lambda_{(z^{\prime})_K}(z) \leq 3\epsilon.
    $$
    By Lemma \ref{smooth approximation}, there is a neighborhood $U_m$ of $X$ and a function $\phi_m \in C^{\infty}(U_m)\cap P_q(U_m)$ such that 
    $
    |\phi_m(z) - f_m(z)| < \epsilon.
    $ for $z \in X$.
    In particular, this implies that 
    $
    |\phi_m + C|z|^2| < 4\epsilon$ for $z \in X(p, r).$
    By shrinking $U_m$ and using the continuity of the two functions,
    $
    |\phi_m(z) + C|z|^2| < {1 \over m},$ for $z \in \overline{\mathbb{B}(p, r)} \cap U_m.
    $
\end{proof}

In the following lemma, we will need the notion of strictly $q$-subharmonic functions for certain $C^2$-functions.  A function $f \in P_q(U) \cap C^2(U)$ if and only if for every $(0, q)$-form $u$,
\begin{equation}\label{q-subharmonic}
\sum_{|K| = q - 1}{\hspace{-8pt}{\vphantom{\sum}}'}\hspace{3pt} \sum_{j, k = 1}^n {\partial^2 f(z) \over \partial z_j \partial \bar{z}_k}u_{jK}\overline{u_{kK}} \geq 0,
\end{equation}
\cite[p.84, p.88]{St10}, \cite[p. 600]{AhDi09}, and we will say that $f$ is $C^2$-strictly $q$-subharmonic, denoted by $f \in C^{2} \cap SP_q(U)$, if the inequality is strict. SPSH will denote the strictly plurisubharmonic functions.
\begin{lem}\cite[Proposition 1.6]{Si87}\label{Proposition 1.6 of Sibony}
Let $X$ be compact in $\mathbb{C}^n$ with Jensen boundary $J_q(X)$.  Let $z_0$ be a point not in the interior of $J_q(X)$.  Then for every $q$-Jensen measure $\mu$ centered at $z_0$, $\mu(\hbox{int }J_q(X)) = 0$.
\end{lem}

\begin{proof}
Suppose $\hbox{int } J_q(X) \neq \emptyset$. Let $p \in X$ and $r > 0$ be such that
$
X(p, r) = \overline{\mathbb{B}(p, r)} \cap X \subseteq \hbox{int } J_q(X).
$
It suffices to show that $\mu(\mathbb{B}(p, {r \over 2}) \cap X) = 0$.  Let $\chi \in C_c^{\infty}(\mathbb{C}^n)$ be nonnegative, identically equal to 1 in a neighborhood of $\partial \mathbb{B}(p, r)$ and 0 in $\mathbb{B}(p, {r \over 2})$.  Let $C$ be a constant such that $C|z|^2 - \chi \in C^{\infty} \cap SPSH(\mathbb{C}^n)$.  By Lemma \ref{A Pq function which is close on the Jensen boundary}, there exists neighborhoods $U_m$ of $X$ and functions $\phi_m \in C^{\infty}(U_m) \cap P_q(U_m)$ such that
\begin{equation}\label{relevant inequalities for phim}
    |\phi_m + C|z|^2| < {1 \over m}, \quad z \in \mathbb{B}(p, r) \cap U_m.
\end{equation}
Since $C|z|^2 - \chi \in SPSH(\mathbb{C}^n) \cap C^\infty(\mathbb{C}^n)$ and $\phi_m \in P_q(U_m)$ 
\begin{equation}\label{phim is strictly q-subharmonic in this location}
    \phi_m + C|z|^2 - \chi + {1 \over k} \in SP_q(\mathbb{B}(p, r) \cap U_m).
\end{equation}
For $m, k > 1$, let 
\begin{equation}\label{definition of psimk}
\psi_{m, k} = \begin{cases}
    \max(0, \phi_m + C|z|^2 - \chi + {1 \over k}) & \mathbb{B}(p, r) \cap U_m
    \\
    0 & U_m \setminus (\mathbb{B}(p, r) \cap U_m).
    \end{cases}
\end{equation}
By \eqref{relevant inequalities for phim}, in $U_m$ and near $\partial \mathbb{B}(p, r)$,
$$
\phi_m + C|z|^2 - \chi + {1 \over k} < {1 \over m} + {1 \over k} - 1 \leq 0, \quad m, k > 1.
$$
Thus, $\psi_{m, k} \in C(U_m)$.  By Lemma \ref{gluing lemma}, $\psi_{m, k} \in P_q(U_m)$.  Notice that on $\mathbb{B}(p, {r \over 2}) \cap U_m$, using \eqref{relevant inequalities for phim},
\begin{equation}\label{the long expression is positive over here}
    \phi_m + C|z|^2 - \chi + {1 \over k} = \phi_m + C|z|^2 + {1 \over k} > -{1 \over m} + {1 \over k} > 0, \quad \hbox{for $k < m$.}
\end{equation}
By \eqref{phim is strictly q-subharmonic in this location}-\eqref{the long expression is positive over here},
\begin{equation}\label{psimk is strictly q-subharmonic in this region which will allow the integral estimate to work}
    \psi_{m, k} \in C^{\infty} \cap SP_q(U_m \cap \mathbb{B}(p, {r \over 2})) \hbox{ for } 1 < k < m.
\end{equation}
Let $\theta \in C^2(\mathbb{C}^n)$ with support on $\mathbb{B}(p, {r \over 2})$.  By \eqref{psimk is strictly q-subharmonic in this region which will allow the integral estimate to work} and since $\psi_{m, k} \in P_q(U_m)$, we have for $\delta > 0$ sufficiently small $\psi_{m, k} + \delta\theta \in P_q(U_m)$; hence $\psi_{m, k} + \delta\theta|_X \in P_q(X)$.  Since $z_0 \not\in \mathbb{B}(p, r)$, $\psi_{m, k}(z_0) = 0$.  Since $\mu$ is a Jensen measure at $z_0$, for $1 < k < m$,
\begin{equation}\label{0 is less than this integral}
0 = \psi_{m, k}(z_0) + \delta\theta(z_0) \leq \int_{X} \psi_{m, k} + \delta\theta d\mu
= \int_{X \cap \mathbb{B}(p, r)} \psi_{m, k} + \delta\theta d\mu. 
\end{equation}
By \eqref{relevant inequalities for phim} and \eqref{definition of psimk},
\begin{equation*}
    \psi_{m, k} < \max(0, {1 \over m} -\chi + {1 \over k}) < {1 \over m} + {1 \over k}, \quad X \cap \mathbb{B}(p, r).
\end{equation*}
Letting $m \to \infty$ and then $k \to \infty$ in \eqref{0 is less than this integral} yields
$$
0 \leq \int_{X \cap \mathbb{B}(p, r)}\theta d\mu = \int_{X \cap \mathbb{B}(p, {r \over 2})}\theta d\mu.
$$
Since $\theta$ is an arbitrary $C^2$ function supported on $\mathbb{B}(p, {r \over 2})$, $\mu(X \cap \mathbb{B}(p, {r \over 2})) = 0$.
\end{proof}

In the next two lemmas, $d(z, F)$ will denote the distance from $z$ to $F$.

\begin{lem}\label{distance function is Pq}
Let $X$ be compact, $F = X \setminus \hbox{int }J_q(X)$ be nonempty and $h(z) = d(z, F)$.  Then $h \in P_q(X)$  Moreover, for any $\epsilon > 0$, there is a neighborhood $W$ of $X$ and a function $\psi \in P_q(W)$ such that 
$
|\psi(z) - h(z)| < \epsilon,
$
for all $z \in W$.

\end{lem}
\begin{proof}
   Since $0 \leq h(z)$ and $h|_F \equiv 0$, $\tilde{h}(z) = 0$, for $z \in F$.  On the other hand by \eqref{Edwards Theorem},
    $
        \tilde{h}(z) = h(z)$ for $z \in X \setminus F$.  Thus, $\tilde{h} \equiv h$ on $X$.  By Lemma \ref{tilde in Pq}, $h \in P_q(X)$. By definition there is a neighborhood $W$ of $X$ and a function $\psi \in P_q(W)$ such that $\psi$ and $h$ are $\epsilon/2$ close on $X$.  Since both functions are continuous on $W$, after possibly shrinking $W$, they are $\epsilon$ close on $W$.
\end{proof}

\begin{lem}\label{break up of the longer lemma into a second part}
    Let $F = X \setminus \hbox{int }J_q(X)$.  If $U$ is a neighborhood of $F$ and $\phi \in P_q(U)$ with $\phi > 0$, then there is a function $\theta$ defined in a neighborhood $W$ of $X$ such that $\theta \in P_q(W)$ and $\theta = \phi$ on a neighborhood of $F$.
\end{lem}
\begin{proof}
After shrinking $U$ we may suppose that $\phi$ is bounded on $U$.  Let $\delta$ be such that 
\begin{equation}\label{how big U is}
\{d(z, F) < 4\delta\} \ssubset U.
\end{equation}
Using Lemma \ref{distance function is Pq}, let $W$ be a neighborhood of $X$ and $\psi_1$ be a function in $P_q(W)$ such that
$|\psi_1(z) - d(z, F)| < \delta$ for all $z \in W$. The set $W \cap \{d(z, F) < \delta\}$ is a nonempty neighborhood of $F$, and $W \cap \{d(z, F) \geq 3.5\delta\}$ is a possibly empty set.
Let $\psi_2:W \to \mathbb{R}$ by $\psi_2 = \psi_1 - 2\delta$.  Then
\begin{eqnarray*}
    \psi_2 = \psi_1 - 2\delta &<& d(z, F) + \delta - 2\delta < 0, \quad z \in W \cap \{d(z, F) < \delta\}.
\end{eqnarray*}
Additionally,
\begin{eqnarray*}
\psi_2(z) = \psi_1(z) - 2\delta &\geq& d(z, F) - \delta - 2\delta \geq .5\delta, \quad W \cap \{d(z, F) \geq 3.5\delta\}.
\end{eqnarray*}
Rescale $\psi_2$ such that
\begin{equation}\label{where psi2 is negative and positive}
\psi_2(z) < 0, \quad z \in \{d(z, F) < \delta\} \cap W, \quad \psi_2(z) > \max_{U}\phi, \quad W \cap \{d(z, F) \geq 3.5\delta\}.
\end{equation}
By the continuity of $\psi_2$, $\psi_2(z) > \max_U \phi$ holds in a neighborhood in $W$ of $W \cap \{d(z, F) \geq 3.5\delta\}$. Define $\theta$ by
\begin{equation}
\theta = \begin{cases} 
\max(\psi_2, \phi), & W\cap \{d(z, F) < 3.5\delta\}
\\
\psi_2 & W \cap \{d(z, F) \geq 3.5\delta\}.
\end{cases}
\end{equation}
By \eqref{how big U is}, $\phi$ is well-defined on the nonempty neighborhood $W\cap \{d(z, F) < 3.5\delta\}$.  If $W \cap \{d(z, F) \geq 3.5\delta\}$ is empty, then using Lemma \ref{closed under maximums}, $\theta \in P_q(W)$.  If it is nonempty, then by Lemma \ref{gluing lemma}, $\theta \in P_q(W)$.  
Since $\phi > 0$, by \eqref{where psi2 is negative and positive},
$
\theta|_{W \cap \{d(z, F) < \delta\}} = \phi.
$
 The proof of the lemma is complete.

\end{proof}

\begin{lem}\cite[Lemme 1.8]{St10}\label{Lemme 1.8 of Sibony}
Let $F = X \setminus \hbox{int }J_q(X)$ be nonempty.  The restrictions of functions in $P_q(X)$ to $F$ are dense in $P_q(F)$.  
\end{lem}

\begin{proof}
    Let $f_1 \in P_q(F)$.  After possibly adding a positive constant, $f_1 > 0$ on $F$. Let $\phi_j \in P_q(G_j)$ where $G_j$ is a neighborhood of $F$ such that $\phi_j > 0$ on $G_j$ and $\phi_j \to f_1$ uniformly on $F$.  By Lemma \ref{break up of the longer lemma into a second part} given $\phi_j \in P_q(G_j)$, there is a neighborhood $W$ of $X$ and a function $\theta_j \in P_q(W)$ such that $\theta_j = \phi_j$ on a neighborhood of $F$.  Since $\theta_j|_X \in P_q(X)$ and $\theta_j \to f_1$, uniformly on $F$, the lemma is proved.

\end{proof}

\begin{pro}\cite[Corollaire 1.7]{Si87}\label{pro:J(F) trivial}
Let $z_0\in F = X \setminus \hbox{int }J_q(X)$ and let $\mu$ be a $q$-Jensen measure on $X$ centered at $z_0$, then $\mu$ is also a $q$-Jensen measure for $z_0$ relative to $P_q(F)$.
\end{pro}
\begin{proof}
By Lemma \ref{Proposition 1.6 of Sibony}, $\mu$ is supported in $F$. Moreover, given $h\in P_q(F)$, by Lemma \ref{Lemme 1.8 of Sibony}, we have that there exists a sequence of $h_j\in P_q(X)$ such that $h_j\vert_F\to h$ uniformly on $F$, so
$$h(z_0)=\lim_{j \to \infty} h_j(z_0)\leq \lim_{j \to \infty}\int_{X}h_jd\mu=\lim_{j \to \infty} \int_{F}h_jd\mu=\int_F hd\mu\;.$$
Therefore, $\mu$ is also a $q$-Jensen measure for $z_0$ relative to $P_q(F)$.
%
%
\end{proof}

\begin{proof}[Proof of Theorem \ref{thm:Countable union of Pq}]
Let, as before, $F=X\setminus\textrm{int}\, J_q(X)$ and suppose, towards a contradiction, that $F\neq \emptyset$. We write $F$ as the union of the closed sets $F\cap X_k$. By the Baire Category Theorem, at least one of these has nonempty interior relative to $F$. We find $p\in F$, $k\in\mathbb{N}$ and $r>0$ such that $F_r=\overline{\mathbb{B}(p,r)}\cap F\subseteq F\cap X_k$. Obviously, $F_r\subseteq J_q(X_k)=X_k$; therefore, $F_r=J_q(F_r)$ (i.e., Property $(P_q)$ holds for $F_r$ because it holds for $X_k$). On the other hand, by \cite[Lemma 4.12]{St10} applied to the compact set $F$, the point $p$ and the radius $r$,
$$F\cap \mathbb{B}(p,r)=F_r\setminus b\mathbb{B}(p,r)=J_q(F_r)\setminus b\mathbb{B}(p,r)=J_q(F\cap \overline{\mathbb{B}(p,r)})\setminus b\mathbb{B}(p,r)\subseteq J_q(F)\;.$$
By Proposition \ref{pro:J(F) trivial}, if $z_0\in J_q(F)$, then $z_0\in J_q(X)$, as the only $q$-Jensen measure for $z_0$ relative to $P_q(X)$ has to be the Dirac delta measure at $z_0$. Hence $F\cap \mathbb{B}(p,r)\subseteq J_q(X)$, but as $(X\setminus F) \cap \mathbb{B}(p,r)\subseteq J_q(X)$, we have that $\mathbb{B}(p,r)\cap X\subseteq J_q(X)$. Hence $p\not\in F$, which is a contradiction.
\end{proof}


\bibliographystyle{amsplain}

\bibliography{bibliography}













\end{document}